\documentclass[reqno]{amsart}
\usepackage{amssymb}
\usepackage{txfonts}
\usepackage{graphics}
\usepackage{latexsym}
\usepackage[dvips]{epsfig}
\usepackage{color}
\usepackage{graphicx}
\usepackage{amsmath}
\usepackage{amssymb}
\usepackage{amsfonts}
\usepackage{eufrak}
\usepackage{amstext}
\usepackage{amsopn}
\usepackage{amsbsy}
\usepackage{amscd}
\usepackage{amsxtra}
\usepackage{amsthm}
\usepackage{bm}
\usepackage{CJK}

\newcommand{\R}{{\mathbb R}}

\newcommand{\beq}{\begin{equation}}
\newcommand{\eeq}{\end{equation}}
\newcommand{\ben}{\begin{eqnarray}}
\newcommand{\een}{\end{eqnarray}}
\newcommand{\beno}{\begin{eqnarray*}}
\newcommand{\eeno}{\end{eqnarray*}}

\newtheorem{thm}{Theorem}[section]

\newtheorem{lem}[thm]{Lemma}
\newtheorem{prop}[thm]{Proposition}
\newtheorem{coro}[thm]{Corollary}

\allowdisplaybreaks

\begin{document}

\title[Finite Morse index solutions]{Some remarks on the structure of finite Morse index solutions to the Allen-Cahn equation in $\R^2$}
\author[K. Wang]{ Kelei Wang}
 \address{\noindent K. Wang-
 School of Mathematics and Statistics
\& Computational Science Hubei Key Laboratory, Wuhan University,
Wuhan, 430072, China.}
\email{wangkelei@whu.edu.cn}

\begin{abstract}
For a solution of the Allen-Cahn equation in $\R^2$, under the
natural linear growth energy bound, we show that the blowing down
limit is unique. Furthermore, if the solution has finite Morse
index, the blowing down limit satisfies the multiplicity one
property.
\end{abstract}



\maketitle

\date{}

\section{Introduction}
\setcounter{equation}{0}

Let $u\in C^2(\R^2)$ be a solution to the problem
\begin{equation}\label{Allen-Cahn}
\Delta u=W^\prime(u)
\end{equation}
where $W$ is a standard double-well potential.

Assume the energy grows linearly, i.e. there exists a constant $C>0$
such that
\begin{equation}\label{linear energy growth}
\int_{B_R(0)}\frac{1}{2}|\nabla u|^2+W(u)\leq CR,\quad \forall R>0.
\end{equation}

For $\varepsilon\to0$, let
\[u_\varepsilon(x,y):=u(\varepsilon^{-1}x,\varepsilon^{-1}y).\]
By \eqref{linear energy growth}, we can assume that, up to a
subsequence of $\varepsilon\to0$,
\[\varepsilon|\nabla
u_\varepsilon|^2dxdy\rightharpoonup\mu_1,\]
\[\frac{1}{\varepsilon}W(u_\varepsilon)dxdy\rightharpoonup\mu_2,\]
weakly as Radon measures on any compact set of $\R^2$. Denote
$\mu=\mu_1/2+\mu_2$ and $\Sigma=\mbox{spt}\mu$.

We can also assume the matrix valued measures
\[\varepsilon\nabla u_\varepsilon\otimes \nabla u_\varepsilon dx\rightharpoonup[\tau_{\alpha\beta}]\mu_1,\]
where $[\tau_{\alpha\beta}]$, $1\leq\alpha,\beta\leq 2$, is
measurable with respect to $\mu_1$. Moreover, $\tau$ is nonnegative
definite $\mu_1$-almost everywhere and it satisfies
\[\sum_{\alpha=1}^2\tau_{\alpha\alpha}=1,\quad \mu_1-a.e.\]

By \cite{H-T}, we have the following characterization about the
convergence of $u_\varepsilon$:
\begin{thm}\label{thm blowing down}
\begin{itemize}
\item[(i)] $u_\varepsilon\to\pm 1$ uniformly on any compact set of
$\R^2\setminus\Sigma$;
\item[(ii)] there exists $N\in\mathbb{N}$ and $N$ unit vectors $e_i, 1\leq i\leq
N$, such that $\Sigma=\cup_{i=1}^NL_i$, where
\[L_i:=\{te_i: t\geq0\};\]
\item[(iii)] $\mu_1=2\mu_2=\sigma_0\sum_{i=1}^Nn_i\mathcal{H}^1\lfloor_{L_i}$, where $\sigma_0$ is a constant and $n_i\in\mathbb{N}$;
\item[(iv)] $I-\tau=e_i\otimes  e_i$ on $L_i\setminus\{0\}$;
\item[(v)] $\sum_{i=1}^Nn_ie_i=0$.
\end{itemize}
\end{thm}
In the above, the constant $\sigma_0$ is determined as follows.
There exists a function $g\in C^2(\R)$ satisfying
\begin{equation}\label{1d profile}
 \left\{\begin{aligned}
&g^{\prime\prime}=W^\prime(g),\ \ \ \mbox{on }\ \R,\\
&g(0)=0,\\
&\lim_{t\to\pm\infty}g(t)=\pm1.
                          \end{aligned} \right .
\end{equation}
Moreover, the following identity holds for $g$:
\begin{equation}\label{first integral}
g^\prime(t)=\sqrt{2W(g(t))}>0,\quad \mbox{on }\ \R.
\end{equation}
As $t\to\pm\infty$, $g(t)$ converges to $\pm 1$ exponentially. Hence
the following quantity is finite:
\[\sigma_0:=\int_{-\infty}^{+\infty}\frac{1}{2}\big|g^\prime(t)\big|^2+W(g(t))dt
=\int_{-\infty}^{+\infty}\big|g^\prime(t)\big|^2dt.\]

In this theorem, we do not claim the uniqueness of $\Sigma$ and
$(n_i)$, because it is obtained by a compactness argument. It may
depend on the subsequence of $\varepsilon\to0$. Our first main
result is
\begin{thm}\label{main result 1}
$\Sigma$ and $(n_1,\cdots, n_N)$ is uniquely determined by $u$.
\end{thm}

Next we further assume that $u$ has finite Morse index, i.e. the
maximal dimension of linear subspaces of
\[\{\varphi\in C_0^\infty(\R^2): \int_{\R^2}|\nabla\varphi|^2+W^{\prime\prime}(u)\varphi^2\leq 0\}\]
is finite. This is equivalent to the fact that $u$ is stable outside
a compact set (see \cite{dev}), i.e. there exists a compact set $K$
such that for any $\varphi\in C_0^\infty(\R^2\setminus K)$,
\[\int_{\R^2}|\nabla\varphi|^2+W^{\prime\prime}(u)\varphi^2\geq 0.\]

Our second result is
\begin{thm}\label{main result 2}
Let $u$ be a solution of \eqref{Allen-Cahn} with finite Morse index.
Then in the blowing down limit, $n_i=1$ for every $i=1,\cdots, N$.
\end{thm}
As in \cite{DKP}, we introduce the following notations. Assume $e_i$
are in clockwise order. For each $i=1,\cdots, N$, let $L_i^\pm$ be
the rays generated by the vector $(e_i+e_{i+1})/2$ and
$(e_i+e_{i-1})/2$ respectively (with obvious modification at the end
points $i=1,N$). Denote $\Omega_i$ to be the cone bounded by
$L_i^\pm$. Our final result says
\begin{thm}\label{main result 3}
Let $u$ be a solution of \eqref{Allen-Cahn} in $\R^2$ with finite
Morse index, and $\Omega_i$ be defined as above. In each $\Omega_i$,
which we assume to be the cone $\{-\lambda_-x<y<\lambda_+ x\}$ for
two positive constants $\lambda_\pm$, there exists three constants
$C$, $R_0$ and $t_i$ such that
\[\sup_{-\lambda_-x<y<\lambda_+x}\big|u(x,y)-g(y-t_i)\big|\leq Ce^{-\frac{x}{C}}, \quad\forall x>R_0.\]
\end{thm}
If we have known Theorem \ref{main result 2}, this theorem will
follow from the refined asymptotic result in \cite{DKP}. Here the
point is, we can prove Theorem \ref{main result 2} and Theorem
\ref{main result 3} at the same time. This will be achieved by
adapting Gui's method in \cite{Gui} to the multiple interfaces
setting.

It should be mentioned that it is conjectured that finite Morse
index solutions of \eqref{Allen-Cahn} satisfies the energy growth
bound \eqref{linear energy growth}. On the other hand, if a solution
satisfies the conclusion of Theorem \ref{main result 3}, it has
finite Morse index (see \cite{KLP 1}).

In this paper, a point in $\R^2$ is denoted by $X=(x,y)$.

The organization of this paper is as follows. In Section 2 we prove
Theorem \ref{main result 1}. Theorem \ref{main result 2} and Theorem
\ref{main result 3} is proved in Section 3 at the same time.

\section{Uniqueness of the blowing down limit}
\setcounter{equation}{0}

By direct integration by parts, we get the stationary condition
\[\int_{\R^2}\left[\frac{1}{2}|\nabla u|^2+W(u)\right]\mbox{div}X-DX(\nabla u,\nabla u)=0,
 \quad \forall X\in C_0^\infty(\R^2,\R^2).\]
 Following \cite{Smrynelis}, this condition implies the
existence of a function $U\in C^3(\R^2)$ satisfying
\begin{align*}
\nabla^2U=
    \begin{bmatrix}
     u_{x}^2-u_{y}^2+2W(u)       & 2u_{x}u_{y}            \\
    2u_{x}u_{y}   & u_{y}^2-u_{x}^2+2W(u)       \\
        \end{bmatrix}.
   \end{align*}
Moreover, by the Modica inequality (see \cite{Modica})
\[\frac{1}{2}|\nabla u|^2\leq W(u),\quad\mbox{in } \R^2,\]
 $U$ is convex. After subtracting
an affine function, we can assume $U(0)=0$ and $\nabla U(0)=0$.
Hence by the convexity of $U$, $U\geq 0$ in $\R^2$.

\begin{lem}
There exists a constant $C$ such that,
\[U(x,y)\leq
C\left(|x|+|y|\right),\quad\mbox{in } \ \R^2.\]
\end{lem}
\begin{proof}
By definition,
\begin{equation}\label{equation for V}
\Delta U=4W(u).
\end{equation}
Then for any $R>0$,
\[\fint_{\partial B_R}U=\int_0^R\frac{d}{dr}\left(\fint_{\partial
B_r}U\right) =\int_0^R \frac{1}{2\pi r}\int_{B_r}4W(u)\leq CR,\]
where we have used \eqref{linear energy growth}.

The conclusion follows from this integral bound and the convexity of
$U$.
\end{proof}

By this linear growth bound and the convexity of $U$, as
$\varepsilon\to0$,
\[U_\varepsilon(x,y):=\varepsilon U(\varepsilon^{-1}x,\varepsilon^{-1}y)\rightarrow U_\infty(x,y)\]
uniformly on compact sets of $\R^2$. Here $U_\infty$ is a
$1$-homogeneous, nonnegative convex function. By convexity, this
limit is independent of subsequences of $\varepsilon\to0$.

Take a sequence $\varepsilon_i\to0$ such that the blowing down limit
of $u_{\varepsilon_i}$ is $\Sigma=\cup_{\alpha=1}^N \{te_\alpha:
t\geq0\}$ and the density on $\{te_\alpha: t\geq0\}$ is $n_\alpha$.
Then outside $\Sigma$, by the strict convexity of $W$ near $\pm1$,
\[|\nabla u_{\varepsilon_i}(X)|^2+W(u_{\varepsilon_i}(X))\leq Ce^{-c \varepsilon_i^{-1} dist(X,\Sigma)}.\]
Because
\begin{align*}
\nabla^2U_{\varepsilon_i}=
    \begin{bmatrix}
     \varepsilon_iu_{\varepsilon_i,x}^2-\varepsilon_iu_{\varepsilon_i,y}^2+\frac{2}{\varepsilon_i}W(u_{\varepsilon_i,})       & 2\varepsilon_iu_{\varepsilon_i,x}u_{\varepsilon_i,y}            \\
    2\varepsilon_iu_{\varepsilon_i,x}u_{\varepsilon_i,y}   & \varepsilon_iu_{\varepsilon_i,y}^2-\varepsilon_iu_{\varepsilon_i,x}^2+\frac{2}{\varepsilon_i}W(u)       \\
        \end{bmatrix},
   \end{align*}
we also have
\[|\nabla^2U_{\varepsilon_i}(X)|^2\leq Ce^{-c \varepsilon_i^{-1} dist(X,\Sigma)}.\]
Hence $\nabla^2U_\infty\equiv 0$ in $\R^2\setminus\Sigma$, that is,
$U_\infty$ is linear in every connected component of
$\R^2\setminus\Sigma$. Thus the set $\{U_\infty<1\}$ is a convex
polygon with its vertex points lying on $\Sigma$. Now it is clear
that $\Sigma$ is uniquely determined by $U_\infty$. Since $U_\infty$
is independent the choice of subsequences of $\varepsilon\to0$,
$\Sigma$ also does not dependent the choice of subsequences of
$\varepsilon\to0$.

In a neighborhood of $\{te_\alpha: t\geq0\}$, written in the
$(e_\alpha,e_\alpha^\bot)$ coordinates, the matrix valued measure
$\nabla^2U_{\varepsilon_i}dxdy$ can be written as
\begin{align*}
\nabla^2U_{\varepsilon_i}dxdy=
    \begin{bmatrix}
     \varepsilon_iu_{\varepsilon_i,e_\alpha}^2-\varepsilon_iu_{\varepsilon_i,e_\alpha^\bot}^2+\frac{2}{\varepsilon_i}W(u_{\varepsilon_i,})       & 2\varepsilon_iu_{\varepsilon_i,e_\alpha}u_{\varepsilon_i,e_\alpha^\bot}            \\
    2\varepsilon_iu_{\varepsilon_i,e_\alpha}u_{\varepsilon_i,e_\alpha^\bot}   & \varepsilon_iu_{\varepsilon_i,e_\alpha^\bot}^2-\varepsilon_iu_{\varepsilon_i,e_\alpha}^2+\frac{2}{\varepsilon_i}W(u)       \\
        \end{bmatrix}dxdy,
   \end{align*}
By Theorem \ref{thm blowing down}, after passing to the limit, we
obtain that in a neighborhood of $\{te_\alpha: t\geq0\}$, the limit
of $\nabla^2U_{\varepsilon_i}dxdy$ equals
\begin{align*}
    \begin{bmatrix}
     0       & 0            \\
   0   & 2n_\alpha\sigma_0\mathcal{H}^1\lfloor_{\{te_\alpha:t\geq0\}}      \\
        \end{bmatrix}.
   \end{align*}
Hence across the ray $\{te_\alpha:t\geq0\}$, $\nabla U_\infty$ has a
jump $2n_\alpha \sigma_0e_\alpha^\bot$. In other words, let
$e^\pm=\nabla U_\infty$ on each side of $\{te_\alpha:t\geq0\}$, then
\[e^+-e^-=2n_\alpha\sigma_0 e_\alpha^\bot.\]
Thus $n_\alpha$ is uniquely determined by $U_\infty$. This proves
Theorem \ref{main result 1}.

\section{The multiplicity one property}
\setcounter{equation}{0}

Since $u$ is assumed to have finite Morse index, it is stable
outside a compact set. Then standard argument using the stable De
Giorgi theorem gives the following
\begin{lem}\label{translation at infinity 0}
For any $X_i=(x_i,y_i)\in u^{-1}(0)\to\infty$,
\[u_i(x,y):=u(x_i+x,y_i+y)\]
converges to a one dimensional solution $g(e\cdot X)$ in
$C^2_{loc}(\R^2)$, where $e$ is a unit vector.
\end{lem}
Recall the cone $\Omega_i$ introduced in Section 1. The nodal set of
$u$ in $\Omega_i$ has the following description.
\begin{lem}\label{nodal sets}
There exists an $R_1>0$ large such that, for each $i$, in
$\Omega_i\setminus B_{R_1}(0)$, $\{u=0\}$ consists of $n_i$ curves,
which can be represented by the graph of functions defined on $L_i$,
with its $C^1$ norm convergeing to $0$ at infinity.
\end{lem}
\begin{proof}
Take an $\Omega_i$, which we assume to be
$\{-\lambda_-x<y<\lambda_+x\}$ for two constants $\lambda_\pm>0$.
$L_i$ is assumed to be the ray $\{x>0,y=0\}$. By \cite[Theorem
5]{Tonegawa}, for all $\varepsilon$ small, there exists a constant
$t_\varepsilon\in(-1/2,1/2)$, such that
\[\{u_\varepsilon=t_\varepsilon\}\cap (B_2\setminus B_{1/2})\cap\Omega_i\]
consists of $n_i$ curves in the form
\[y=h_\varepsilon^\alpha(x), \quad\mbox{for } 1/2\leq x\leq 2,\quad 1\leq\alpha\leq n_i,\]
where $\|h_\varepsilon^\alpha\|_{C^{1,1/2}[1/2,2]}$ is uniformly
bounded. By \cite{H-T}, for each $\alpha$, $h_\varepsilon^\alpha$
converges to $0$ uniformly on $[1/2,2]$ as $\varepsilon\to0$.

By Lemma \ref{translation at infinity 0}, for each $t\in[-3/4,3/4]$,
$\{u_\varepsilon=t\}$ consists of $n_i$ curves, in the form
\[y=h_\varepsilon^\alpha(x,t), \quad\mbox{for } 1/2\leq x\leq
2,\quad 1\leq\alpha\leq n_i,\]
 which lies in an $O(\varepsilon)$ neighborhood of
 $\{u_\varepsilon=t_\varepsilon\}$. Moreover, after a scaling and
 using Lemma \ref{translation at infinity 0}, we get
 \[\lim_{\varepsilon\to0}\sup_{1/2\leq x\leq 2}\Big|\frac{d}{dx}h_\varepsilon^\alpha(x,t)\Big|=0.\]
Rescaling back to $u$ we conclude the proof.
\end{proof}

Now we are in the following situation:
\begin{enumerate}
\item[{\em (H1)}] There are two positive constants $R>0$ large and
$\lambda>0$.
\item[{\em (H2)}] The domain $\mathcal{C}:=\{(x,y): |y|<\lambda x,
x>R\}$.
\item[{\em (H3)}] $u\in C^2(\overline{\mathcal {C}})$ satisfies
\eqref{Allen-Cahn} in $\mathcal{C}$.
\item[{\em (H4)}] $\{u=0\}$ consists of $N$ curves
$\{y=f_i(x)\}$, $1\leq i\leq N$, where $f_i\in C^\infty[R,+\infty)$
satisfying
\[f_1<f_2<\cdots<f_N,\]
\[\lim_{x\to+\infty}f_i^\prime(x)=0,\quad\forall 1\leq i\leq N.\]
\end{enumerate}
The last condition implies that
\[\lim_{x\to+\infty}\frac{|f_i(x)|}{|x|}=0,\quad\forall 1\leq i\leq N.\]

The main goal in this section is to prove
\begin{thm}\label{thm 3.1}
We must have $N=1$. Moreover, there exists a constant $t$ such that
\[\big|f(x)-t\big|\leq Ce^{-\frac{x}{C}},\]
and
\[\sup_{-\lambda x<y<\lambda x}\big|u(x,y)-g(y-t)|\leq Ce^{-\frac{x}{C}},\]
 where the constant $C$ depends only on $W$.
\end{thm}
Theorem \ref{main result 2} and \ref{main result 3} follow from this
theorem, Theorem \ref{thm blowing down}, Theorem \ref{main result 1}
and Lemma \ref{nodal sets}.

Possibly by a change of sign, assume $u<0$ in $\{y<f_1(x)\}$.

\begin{lem}\label{translation at infinity}
For any $1\leq i\leq N$ and $t\to+\infty$,
\[u^t(x,y):=u(t+x, f_i(t)+y)\]
converges to $g(y)$ in $C^2_{loc}(\R^2)$.
\end{lem}
\begin{proof}
This is a consequence of Lemma \ref{translation at infinity 0} and
Lemma \ref{nodal sets}. Note that $\{u^t=0\}=\{y=f^t(x)\}$ where
$f^t(x):=f_i(x+t)-f_i(t)$. As $t\to+\infty$, $\frac{df^t}{dx}$
converges to $0$ uniformly on any compact set of $\R$. Hence by
noting that $f^t(0)=0$, $f^t$ also converges to $0$ uniformly on any
compact set of $\R$. This implies that the limit $u^\infty=0$ on
$\{y=0\}$. From this we see $u_\infty(x,y)\equiv g(y)$. Since this
limit is independent of subsequences of $t\to+\infty$, we finish the
proof.
\end{proof}

\begin{lem}\label{exponential decay}
In $\overline{\mathcal{C}}$,
\[1-u(x,y)^2\leq Ce^{-c\min_i(y-f_i(x))}.\]
\end{lem}
\begin{proof}
By the previous lemma, for any $M>0$, if we have chosen $R$ large
enough, $u^2>1-\sigma(M)$ in $\{(x,y): |y-f_i(x)|>M, \forall i\}$,
where $\sigma(M)$ is a constant depending on $M$ satisfying
$\lim_{M\to+\infty}\sigma(M)=0$. By choosing $M$ large (then
$\sigma(M)$ can be made small so that $W$ is strictly convex in
$(1-\sigma(M),1)$), in $\{(x,y): |y-f_i(x)|>M,\forall i\}$,
\[\Delta W(u)\geq cW(u).\]
From this we deduce the exponential decay
\[W(u)\leq
Ce^{-c dist(X, \cup_i\{(x,y): |y-f_i(x)|<M\})}.\]
 Finally,
because $|f_i^\prime(x)|<1$, the distance to $\{y=f_i(x)\}$ is
comparable to $|y-f_i(x)|$. This finishes the proof.
\end{proof}
As a consequence,
\begin{equation}\label{5.1}
1-u(x,y)^2\sim O(e^{-cx})\quad \mbox{on}\ \{y=\pm\lambda x\}.
\end{equation}

Another consequence of this exponential decay is:
\begin{coro}\label{exponential decay derivatives}
In $\mathcal{C}$,
\[|u_x(x,y)|+|u_{xx}(x,y)|\leq Ce^{-\frac{\min_i(y-f_i(x))}{C}}.\]
\end{coro}
This follows from standard gradient estimates.

This exponential decay implies that
\begin{equation}\label{5.01}
\int_{-\lambda x}^{\lambda x}u_{x}(x,y)^2+u_{xx}(x,y)^2dy\leq
C,\quad\forall x>R.
\end{equation}

\begin{lem}\label{lem 3.3}
For any $1\leq i\leq N-1$,
\[\lim_{x\to+\infty}\left(f_{i+1}(x)-f_i(x)\right)=+\infty.\]
\end{lem}
\begin{proof}
By Lemma \ref{translation at infinity}, for any $t\to+\infty$,
\[u^t(x,y):=u(x+t,y+f_i(t))\]
converges uniformly to $g(y)$ on any compact set of $\R^2$.

From this we see, for any $L>0$, if $t$ is large enough, $u^t>0$ on
$\{x=0,0<y<L\}$ and $u^t<0$ on $\{x=0,-L<y<0\}$. The conclusion
follows from this claim directly.
\end{proof}

\begin{prop}\label{Hamiltonian identity 2}
For any $x>R$,
\[\int_{-\lambda x}^{\lambda x}\frac{u_y^2-u_x^2}{2}+W(u)dy=N\sigma_0+O(e^{-cx}). \]
\end{prop}
\begin{proof}
This is the Hamiltonian identity, see \cite{Gui}.

First, differentiating in $x$, integrating by parts and using
\eqref{5.1} leads to
\begin{equation}\label{3.5}
\frac{d}{dx}\int_{-\lambda x}^{\lambda
x}\frac{u_y^2-u_x^2}{2}+W(u)dy=O(e^{-cx}).
 \end{equation}
 Next, by
Lemma \ref{exponential decay}, for any $\delta>0$, there exists an
$L>0$ such that for all $x$,
\begin{equation}\label{3.3}
\int_{\{y\in(-\lambda x,\lambda x): |y-f_i(x)|>L,\forall
i\}}\frac{u_y^2-u_x^2}{2}+W(u)dy\leq\delta.
\end{equation}
While for each $i=1,\cdots, N$, by Lemma \ref{translation at
infinity}, we have
\begin{equation}\label{3.4}
\lim_{x\to+\infty}\int_{f_i(x)-L}^{f_i(x)+L}\frac{u_y^2-u_x^2}{2}+W(u)dy
=\int_{-L}^L\frac{1}{2}g^\prime(y)^2+W(g(y))dy=\sigma_0+O(\delta),
\end{equation}
where in the last step we have used the exponential convergence of
$g$ at infinity.

Combining \eqref{3.3} and \eqref{3.4}, by noting that $\delta$ can
be arbitrarily small, we get
\[\lim_{x\to+\infty}\int_{-\lambda x}^{\lambda
x}\frac{u_y^2-u_x^2}{2}+W(u)dy=N\sigma_0.\] The conclusion of this
lemma follows by combining this identity and \eqref{3.5}.
\end{proof}

\begin{prop}\label{second eigenvalue}
There exist two constants $L_0>0$ and $\mu>0$ so that the following
holds. For any constants $L^+>L_0$ and $L^->L_0$ and $v\in
H^1(-L^-,L^+)$ satisfying
\begin{equation}\label{constraint}
\int_{-L^-}^{L^+}v(t)g^\prime(t)dt=0,
\end{equation}
we have
\begin{equation}\label{positive second eigenvalue}
\int_{-L^-}^{L^+}\Big|\frac{dv}{dt}(t)\Big|^2+W^{\prime\prime}(g(t))v(t)^2dt\geq\mu\int_{-L^-}^{L^+}v(t)^2dt.
\end{equation}
\end{prop}
\begin{proof}
Assume by the contrary, there exist $L_j^\pm\to+\infty$ and $v_j\in
H^1(-L^-_j,L^+_j)$ satisfying
\begin{equation}\label{3.6}
\int_{-L_j^-}^{L_j^+}v_j(t)g^\prime(t)dt=0,
\end{equation}
and
\begin{equation}\label{3.7}
\int_{-L_j^-}^{L_j^+}v_j(t)^2dt=1,
\end{equation}
but
\begin{equation}\label{3.8}
\int_{-L_j^-}^{L_j^+}\Big|\frac{dv_j}{dt}(t)\Big|^2+W^{\prime\prime}(g(t))v_j(t)^2dt
\leq\frac{1}{j}.
\end{equation}

From the last two assumptions we deduce that
\begin{equation}\label{3.9}
\int_{-L_j^-}^{L_j^+}\Big|\frac{dv_j}{dt}(t)\Big|^2dt \leq C,
\end{equation}
for some constant $C$ depending only on $\sup|W^{\prime\prime}|$.
Hence the $1/2$-H\"{o}lder seminorm of $v_j$ is uniformly bounded.
Then by \eqref{3.7}, $\sup|v_j|$ is also uniformly bounded. Assume
$v_j$ converges to $v_\infty$ in $C_{loc}(\R)$.

By the exponential decay of $g^\prime$ at infinity, \eqref{3.6} can
be passed to the limit, which gives
\begin{equation}\label{3.6.1}
\int_{-\infty}^{+\infty}v_\infty(t)g^\prime(t)dt=0.
\end{equation}
\eqref{3.7} and \eqref{3.9} can also be passed to the limit, leading
to
\begin{equation}\label{3.7.1}
\int_{-\infty}^{+\infty}v_\infty(t)^2+\Big|\frac{dv_\infty}{dt}(t)\Big|^2dt
\leq C+1.
\end{equation}

Because $g$ converges to $\pm1$ at $\pm\infty$ respectively, there
exists an $R_2$ such that
\begin{equation}\label{3.11}
W^{\prime\prime}(g(t))\geq
c_0:=\frac{1}{2}\min\{W^{\prime\prime}(-1), W^{\prime\prime}(1)\}>0,
\quad\mbox{in } \{|t|\geq R_2\}.
\end{equation}
 Thus for any $R\geq R_2$,
\begin{eqnarray*}
\int_{-R}^R\Big|\frac{dv_\infty}{dt}(t)\Big|^2+W^{\prime\prime}(g(t))v_\infty(t)^2dt
&\leq&\liminf_{j\to+\infty}\int_{-R}^R\Big|\frac{dv_j}{dt}(t)\Big|^2+W^{\prime\prime}(g(t))v_j(t)^2dt\\
&\leq&\liminf_{j\to+\infty}\int_{-L_j^-}^{L_j^+}\Big|\frac{dv_j}{dt}(t)\Big|^2+W^{\prime\prime}(g(t))v_j(t)^2dt\\
&\leq&0.
\end{eqnarray*}
By \eqref{3.7.1}, we can let $R\to+\infty$, which leads to
\[\int_{-\infty}^{+\infty}\Big|\frac{dv_\infty}{dt}(t)\Big|^2+W^{\prime\prime}(g(t))v_\infty(t)^2dt\leq0.\]
Then by the spectrum theory for
$-\frac{d^2}{dt^2}+W^{\prime\prime}(g(t))$ (see for example
\cite[Lemma 1.1]{DKP}) and \eqref{3.6.1}, $v_\infty\equiv 0$.

By the convergence of $v_j$ in $C_{loc}(\R)$,
\begin{equation}\label{3.10}
\lim_{j\to+\infty}\int_{-R_2}^{R_2}v_j(t)^2dt=0.
\end{equation}
Substituting this into \eqref{3.8}, by noting \eqref{3.11}, we get
\[
\int_{(-L_j^-,-R_2)\cup(R_2,L_j^+)}v_j(t)^2dt\leq
C\left(\frac{1}{j}+ \int_{-R_2}^{R_2}v_j(t)^2dt\right)\to0.\]
Combining this with \eqref{3.10} we get a contradiction with
\eqref{3.7}. Thus under the assumptions \eqref{3.6} and \eqref{3.7},
\eqref{3.8} cannot hold.
\end{proof}

With these preliminaries, we come to the proof of Theorem \ref{thm
3.1}.
\begin{proof}[Proof of Theorem \ref{thm 3.1}]
Given a tuple $(t_1,\cdots, t_N)$ with $t_1<\cdots<t_N$, define

\makeatletter
\let\@@@alph\@alph
\def\@alph#1{\ifcase#1\or \or $'$\or $''$\fi}\makeatother
\begin{equation}\label{1d solution}
{g(y;t_1,\cdots, t_N)=}
\begin{cases}
g(y-t_1), &y<t_1^+, \nonumber\\
\min\{g(y-t_1),-g(y-t_2)\}, &t_1^+=t_2^-<x<t_2^+,\\
\min\{-g(y-t_2),g(y-t_3)\}, &t_2^+=t_3^-<x<t_3^+,\nonumber\\
\cdots.\nonumber
\end{cases}
\end{equation}
\makeatletter\let\@alph\@@@alph\makeatother In the above,
\[t_i^+:=\frac{t_i+t_{i+1}}{2},\quad t_i^-:=\frac{t_{i-1}+t_i}{2},\]
and for simplicity of notation $t_1^-=-\lambda x$ and $t_N^+=\lambda
x$.

Note that $g(y;t_i)$ is continuous, while its derivative in $t$ has
a jump at $t_i^+$. (In fact, the left and right derivatives at each
$t_i^+$ only differ by a sign.)

Next we define
\[F(x;t_1,\cdots, t_N):=\int_{-\lambda x}^{\lambda x}\big|u(x,y)-g(y;t_1,\cdots, t_N)\big|^2dy.\]

We divide the proof into three steps.

{\bf Step 1.} As $x\to+\infty$, $\int_{-\lambda x}^{\lambda
x}\big|u(x,y)-g(y;f_i(x))\big|^2dy\to0$.

This follows from Lemma \ref{translation at infinity} and Lemma
\ref{exponential decay}.

{\bf Step 2.} By Step 1,
\[\lim_{x\to+\infty}F(x;f_1(x),\cdots,f_N(x))=0.\]
 Moreover, for any $\varepsilon>0$, there exists
a $\delta>0$ such that, if $|t_i-f_i(x)|>\delta$ for some $i$, then
\begin{equation}\label{3.1}
\liminf_{x\to+\infty}F(x;t_1,\cdots, t_N)\geq\varepsilon.
\end{equation}

Direct calculations give
\begin{equation}\label{F derivative}
\frac{\partial F}{\partial t_i}(x;t_1,\cdots,
t_N)=2(-1)^i\int_{t_i^-}^{t_i^+}\left[u(x,y)-(-1)^{i-1}g(y-t_i)\right]g^\prime(y-t_i)dy.
\end{equation}
\begin{eqnarray}\label{F derivative second order}
\frac{\partial^2 F}{\partial t_i^2}(x;t_1,\cdots,
t_N)&=&2\int_{t_i^-}^{t_i^+}g^\prime(y-t_i)^2dy\nonumber\\
&&+2(-1)^{i+1}\int_{t_i^-}^{t_i^+}\left[u(x,y)-(-1)^{i-1}g(y-t_i)\right]g^{\prime\prime}(y-t_i)dy \nonumber\\
&&+O(e^{-c\min\{t_i-t_{i-1},t_{i+1}-t_i\}}).
\end{eqnarray}
By Step 1, Lemma \ref{exponential decay} and the exponential decay
of $g^{\prime\prime}$ at infinity, there exists a $\sigma>0$ such
that, for any $(t_1,\cdots, t_N)$ satisfying $|t_i-f_i(x)|<\sigma$,
$\frac{\partial^2F}{\partial t_i^2}(x;t_i)>\sigma$.

Finally, if $|i-j|>1$, $\frac{\partial^2 F}{\partial t_i\partial
t_j}(x;t_i)=0$ and
\[\Big|\frac{\partial^2 F}{\partial t_i \partial t_{i+1}}(x;t_i)\Big|\leq Ce^{-c\left(t_{i+1}-t_i\right)}.\]
Combining this with \eqref{F derivative second order} we see
$[\frac{\partial^2 F}{\partial t_i \partial t_j}(x;t_i)]$ is
positively definite for those $(t_1,\cdots, t_N)$ satisfying the
condition that $|t_i-f_i(x)|$ is small enough for all $i$.

Combining the above analysis, we see for all $x$ large, there exists
a unique tuple $(t_i(x))$ such that
\[F(x;t_i(x))=\min_{(t_i)\in\R^N}F(x,t_i).\]
Moreover,
\begin{equation}\label{5.1.1}
\lim_{x\to+\infty}|t_i(x)-f_i(x)|=0,\quad \forall 1\leq i\leq N.
\end{equation}
By the implicit function theorem, for each $i$, $t_i(x)$ is twice
differentiable in $x$.

Lemma \ref{lem 3.3} and \eqref{5.1.1} implies that for any $1\leq
i\leq N-1$,
\begin{equation}\label{distance large}
t_{i+1}(x)-t_i(x)\to+\infty,\quad \mbox{as } x\to+\infty.
\end{equation}

Let
\[v(x,y):=u(x,y)-g(y;t_i(x)).\]
Clearly
\begin{equation}\label{5.2}
\lim_{x\to+\infty}\|v\|_{L^2(-\lambda x,\lambda
x)}=\lim_{x\to+\infty}F(x;t_i(x))=0.
\end{equation}

In the following we denote $g^\ast:=g(y;t_i(x))$ and
\[g_i(y):=(-1)^{i-1}g(y-t_i(x)),\quad \mbox{for}\ y\in(t_i^-,t_i^+).\]

By definition,
\begin{equation}\label{3.2}
0=\frac{\partial F}{\partial
t_i}(x;t_i(x))=2\int_{t_i^-(x)}^{t_i^+(x)}
\left(u-g_i\right)g_i^\prime.
\end{equation}

Differentiating \eqref{3.2} with respect to $x$ leads to
\begin{eqnarray}\label{5.4}
&&\left[\int_{t_i^-(x)}^{t_i^+(x)}|g_i^\prime|^2-\left(u-g_i\right)
g_i^{\prime\prime}\right]t_i^\prime(x)+\int_{t_i^-(x)}^{t_i^+(x)}u_xg_i^\prime\nonumber\\
&=&-\left[u(x,t_i^+(x))-g_i(t_i^+(x))\right]g_i^\prime(t_i^+(x))\frac{t_i^\prime(x)+t_{i+1}^\prime(x)}{2}\\
&+&\left[u(x,t_i^-(x))-g_i(t_i^-(x))\right]g_i^\prime(t_i^-(x))\frac{t_i^\prime(x)+t_{i-1}^\prime(x)}{2}.\nonumber
\end{eqnarray}

Note that by the result in Step 1 and the exponential decay of
$g^{\prime\prime}$ at infinity,
\begin{eqnarray*}
\lim_{x\to+\infty}\int_{t_i^-(x)}^{t_i^+(x)}\left(u-g_i\right)
g^{\prime\prime}&\leq&\lim_{x\to+\infty}\left[\int_{t_i^-(x)}^{t_i^+(x)}\left(u-g_i\right)^2\right]^{\frac{1}{2}}
\left[\int_{t_i^-(x)}^{t_i^+(x)}\big|g^{\prime\prime}\big|^2\right]^{1/2}\\
&=&0,
\end{eqnarray*}
while by \eqref{distance large}, there exists a constant $c>0$ such
that
\[\int_{t_i^-(x)}^{t_i^+(x)}|g_i^\prime|^2\geq c, \quad\forall x \mbox{ large}.\]
By Lemma \ref{exponential decay} and \eqref{distance large},
$u(x,t_i^\pm(x))$ and $g_i(t_i^\pm(x))$ all converge to $0$ as
$x\to+\infty$. Thus by \eqref{5.4} we obtain
\begin{equation}\label{5.4.0}
t_i^\prime(x)=-\frac{\int_{t_i^-(x)}^{t_i^+(x)}u_xg^\prime_i}
{\left[\int_{t_i^-(x)}^{t_i^+(x)}|g^\prime_i|^2\right]+o(1)}+o(1)\sum_{j\neq
i}\frac{\Big|\int_{t_j^-(x)}^{t_j^+(x)}u_xg^\prime_j\Big|}
{\left[\int_{t_j^-(x)}^{t_j^+(x)}|g^\prime_j|^2\right]+o(1)}+O(e^{-cx})\rightarrow
0,\quad \mbox{as}\ x\to+\infty.
\end{equation}
Differentiating this once again we see $t_i^{\prime\prime}(x)$ also
converges to $0$ as $x\to+\infty$.

Similar to the calculation in \cite[page 927]{Gui}, we have
\begin{eqnarray*}
&&\int_{t_i^-(x)}^{t_i^+(x)}\left(\frac{u_y^2-u_x^2}{2}+W(u)\right)
-\frac{|g^\prime_i|^2}{2}-W(g_i)\\
&=&\int_{t_i^-(x)}^{t_i^+(x)}\left(\frac{u_y^2-|g_i^\prime|^2}{2}+W(u)-W(g_i)-\frac{u_x^2}{2}\right)\\
&=&\int_{t_i^-(x)}^{t_i^+(x)}\left[W(u)-W(g_i)-\frac{W^\prime(u)+W^\prime(g_i)}{2}\left(u-g_i\right)\right]\\
&&+\frac{1}{2}\int_{t_i^-(x)}^{t_i^+(x)}\left[\left(u-g_i\right)u_{xx}-u_x^2\right]+\mathcal{B},
\end{eqnarray*}
where $\mathcal{B}$ is the boundary terms coming from integrating by
parts. In the above we have used
\begin{eqnarray*}
\int_{t_i^-(x)}^{t_i^+(x)}u_y^2-|g_i^\prime|^2
&=&\int_{t_i^-(x)}^{t_i^+(x)}\left(u_y-g_i^\prime\right)
\left(u_y+g_i^\prime\right)\\
&=&-\int_{t_i^-(x)}^{t_i^+(x)}\left(u-g_i\right)
\left(u_{yy}+g_i^{\prime\prime}\right)\\
&&+\left[u(x,t_i^+(x))-g_i(t_i^+(x))\right]
\left[u_y(x,t_i^+(x))+g_i^\prime(t_i^+(x))\right]\\
&&-\left[u(x,t_i^-(x))-g_i(t_i^-(x))\right]
\left[u_y(x,t_i^-(x))+g_i^\prime(t_i^-(x))\right]\\
&=&-\int_{t_i^-(x)}^{t_i^+(x)}\left(u-g_i\right)
\left[W^\prime(u)+W^\prime(g_i)\right]+\int_{t_i^-(x)}^{t_i^+(x)}u_{xx}\left(u-g_i\right)\\
&&+\left[u(x,t_i^+(x))-g_i(t_i^+(x))\right]
\left[u_y(x,t_i^+(x))+g_i^\prime(t_i^+(x))\right]\\
&&-\left[u(x,t_i^-(x))-g_i(t_i^-(x))\right]
\left[u_y(x,t_i^-(x))+g_i^\prime(t_i^-(x))\right].
\end{eqnarray*}

Summing in $i$ and using the Hamiltonian identity, we obtain
\begin{eqnarray}\label{5.5}
\int_{-\lambda x}^{\lambda x}u_{xx}\left(u-g^\ast\right)-u_x^2
&=&\sum_i\int_{t_i^-(x)}^{t_i^+(x)}\left[\left(u-g_i\right)u_{xx}-u_x^2\right]\nonumber\\
&=&-2\sum_i\left[u(x,t_i^+(x))-g_i(t_i^+(x))\right]g_i^\prime(t_i^+(x))+o(\|v\|^2) \nonumber\\
&&+2\sum_i\left[\int_{t_i^+(x)}^{+\infty}|g_i^\prime|^2+\int_{-\infty}^{t_i^-(x)}|g_i^\prime|^2\right]+O(e^{-cx}).
\end{eqnarray}

On the other hand, similar to \cite[Eq. (4.35)]{Gui}, we have
\begin{eqnarray}\label{5.8}
\int_{t_i^-(x)}^{t_i^+(x)}u_{xx}\left(u-g_i\right)&=&\int_{t_i^-(x)}^{t_i^+(x)}
\left(W^\prime(u)-u_{yy}\right)\left(u-g_i\right)\nonumber\\
&=&\int_{t_i^-(x)}^{t_i^+(x)}\left[W^\prime(u)-W^\prime(g_i)-W^{\prime\prime}(g_i)\left(u-g_i\right)\right]
\left(u-g_i\right)\\
&&+\int_{t_i^-(x)}^{t_i^+(x)}\left(g_i^{\prime\prime}-u_{yy}\right)\left(u-g_i\right)
+W^{\prime\prime}(g_i)\left(u-g_i\right)^2\nonumber\\
&=&o(\|v\|^2)+\int_{t_i^-(x)}^{t_i^+(x)}\big|\left(u-g_i\right)_y\big|^2+W^{\prime\prime}(g_i)\left(u-g_i\right)^2\nonumber\\
&-&\left[u(x,t_i^+(x))-g_i(t_i^+(x))\right]\left[u_y(x,t_i^+(x))-g_i^\prime(t_i^+(x))\right]\nonumber\\
&+&\left[u(x,t_i^-(x))-g_i(t_i^-(x))\right]\left[u_y(x,t_i^-(x))-g_i^\prime(t_i^-(x))\right].\nonumber
\end{eqnarray}

Summing in $i$ we get
\begin{eqnarray}\label{5.8.1}
\int_{-\lambda x}^{\lambda x}u_{xx}\left(u-g^\ast\right)&=&
o(\|v\|^2)+\sum_i\int_{t_i^-(x)}^{t_i^+(x)}\big|\left(u-g_i\right)_y\big|^2+W^{\prime\prime}(g_i)\left(u-g_i\right)^2\nonumber\\
&+&2\sum_i\left[u(x,t_i^+(x))-g_i(t_i^+(x))\right]g_i^\prime(t_i^+(x))+O(e^{-cx}).
\end{eqnarray}

By \eqref{3.2} and \eqref{distance large}, Proposition \ref{second
eigenvalue} applies to $u-g_i$ in $(t_i^-(x),t_i^+(x))$, which gives
\begin{equation}\label{5.7}
\int_{t_i^-(x)}^{t_i^+(x)}\big|\left(u-g_i\right)_y\big|^2+W^{\prime\prime}(g_i)\left(u-g_i\right)^2
\geq\mu\int_{t_i^-(x)}^{t_i^+(x)}\left(u-g_i\right)^2.
\end{equation}

Hence
\begin{equation}\label{5.7.1}
\int_{-\lambda x}^{\lambda x}u_{xx}\left(u-g^\ast\right)
\geq\left(\mu+o(1)\right)\|v\|^2
+2\sum_i\left[u(x,t_i^+(x))-g_i(t_i^+(x))\right]g^\prime_i(t_i^+(x))+O(e^{-cx}).
\end{equation}
Combining this with \eqref{5.5}, we deduce that
\begin{eqnarray}\label{5.7.2}
 \int_{-\lambda
x}^{\lambda x}u_x^2&\geq&\left(\mu+o(1)\right)\|v\|^2
+4\sum_i\left[u(x,t_i^+(x))-g_i(t_i^+(x))\right]g^\prime_i(t_i^+(x))\nonumber\\
&-&2\sum_i\left[\int_{t_i^+(x)}^{+\infty}|g_i^\prime|^2+\int_{-\infty}^{t_i^-(x)}|g_i^\prime|^2\right]+O(e^{-cx}).
\end{eqnarray}

Differentiating $\|v\|^2$ twice in $x$ leads to
\begin{eqnarray*}
\frac{1}{2}\frac{d}{dx}\|v\|^2&=&\sum_i\int_{t_i^-(x)}^{t_i^+(x)}\left(u-g_i\right)\left[u_x+g_i^\prime
t_i^\prime(x)\right]\\
&=&\sum_i\int_{t_i^-(x)}^{t_i^+(x)}\left(u-g_i\right)u_x,\quad{\mbox{(by
\eqref{3.2})}}
\end{eqnarray*}
and
\begin{eqnarray*}
\frac{1}{2}\frac{d^2}{dx^2}\|v\|^2
&=&\sum_i\int_{t_i^-(x)}^{t_i^+(x)} u_x^2+u_xg_i^\prime
t_i^\prime(x)+u_{xx}\left(u-g_i\right)\\
&\geq&2\sum_i\int_{t_i^-(x)}^{t_i^+(x)}u_x^2-\frac{3}{2}\sum_i\frac{\left(\int_{t_i^-(x)}^{t_i^+(x)}
u_xg_i^\prime\right)^2}{\int_{t_i^-(x)}^{t_i^+(x)}\big|g_i^\prime\big|^2}
\quad\mbox{(by \eqref{5.5} and \eqref{5.4.0})}\\
&-&2\sum_i\left[u(x,t_i^+(x))-g_i(t_i^+(x))\right]g_i^\prime(t_i^+(x))+2\sum_i\left[\int_{t_i^+(x)}^{+\infty}|g_i^\prime|^2+\int_{-\infty}^{t_i^-(x)}|g_i^\prime|^2\right]\\
&\geq&\sum_i\frac{1}{2}\int_{t_i^-(x)}^{t_i^+(x)}u_x^2-2\sum_i\left[u(x,t_i^+(x))-g_i(t_i^+(x))\right]g_i^\prime(t_i^+(x))
\quad\mbox{(by Cauchy-Schwarz)}\\
&+&\sum_i\left[\int_{t_i^+(x)}^{+\infty}|g_i^\prime|^2+\int_{-\infty}^{t_i^-(x)}|g_i^\prime|^2\right]\\
&\geq&\frac{1}{2}\left[\mu+o(1)\right]\|v\|^2. \quad\mbox{(by
\eqref{5.7.2})}
\end{eqnarray*}

By noting \eqref{5.2}, from this inequality we deduce that
\begin{equation}\label{5.9}
\|v\|^2\leq Ce^{-cx},\quad\mbox{for all $x$ large}.
\end{equation}

{\bf Step 3.} Note that
\begin{eqnarray*}
g_i(t_i^+(x))g_i^\prime(t_i^+(x))&=&\int_{t_i^+(x)}^{+\infty}\big|g_i^\prime\big|^2+g_ig_i^{\prime\prime}\\
&\leq&\int_{t_i^+(x)}^{+\infty}\big|g_i^\prime\big|^2,
\end{eqnarray*}
because $g_i$ is close to $1$ in $(t_i^+(x),+\infty)$ (see
\eqref{distance large}) and hence
$g_i^{\prime\prime}=W^\prime(g_i)<0$ in this interval. We also have
$g_i(t_i^+(x))g_i^\prime(t_i^+(x))>0$, because $g_i(t_i^+(x))>0$ and
$g_i^\prime(t_i^+(x))>0$.

Then for all $x$ large, by noting that $g_i(t_i^+(x))$ is close to
$1$ and $u(x,t_i^+(x))-g_i(t_i^+(x))$ is close to $0$, we obtain
\begin{eqnarray*}
\Big|\left[u(x,t_i^+(x))-g_i(t_i^+(x))\right]g_i^\prime(t_i^+(x))\Big|&\leq&\frac{1}{2}g_i(t_i^+(x))g_i^\prime(t_i^+(x))\\
&\leq&\frac{1}{2}\int_{t_i^+(x)}^{+\infty}\big|g_i^\prime\big|^2.
\end{eqnarray*}
Substituting this into \eqref{5.5}, we get
\begin{eqnarray}\label{5.10}
\int_{-\lambda x}^{\lambda x}u_x^2&\leq&\int_{-\lambda x}^{\lambda
x}u_{xx}\left(u-g_\ast\right)+o(\|v\|^2)+O(e^{-cx}) \nonumber\\
&\leq&\left[\int_{-\lambda x}^{\lambda
x}u_{xx}^2\right]^{\frac{1}{2}}\|v\|+o(\|v\|^2)+O(e^{-cx})\\
&\leq &Ce^{-cx}.\nonumber \quad\mbox{(by \eqref{5.01} and
\eqref{5.9})}
\end{eqnarray}
Then by \eqref{5.4.0} and the Cauchy-Schwarz inequality, we get
\[|t_i^\prime(x)|\leq Ce^{-cx}, \quad\forall i.\]
Thus for all $1\leq i\leq N$, $\lim_{x\to+\infty}t_i(x)$ exists and
it is finite. By noting \eqref{5.1.1}, for each $i$, the limit
$\lim_{x\to+\infty}f_i(x)$ also exists. In particular, this limit is
finite. Then for all $1\leq i\leq N-1$,
\[\lim_{x\to+\infty}\left(f_{i+1}(x)-f_i(x)\right)\]
also exists and it is finite. However, this is a contradiction with
Lemma \ref{lem 3.3} if $N\geq 2$. Hence we must have $N=1$.

Finally, the exponential convergence of $u(x,\cdot)$ follows from
\eqref{5.10}, and the exponential convergence of $f_i(x)$ follows
from this exponential convergence and the (uniform) positive lower
bound on $g^\prime$ and $u_y(x,\cdot)$ in the part where $|g|<1/2$
and $|u|<1/2$.
\end{proof}

\end{document}